\documentclass[11pt]{article}

\usepackage[margin=1in]{geometry}
\usepackage{amsmath,amssymb,amsthm}
\usepackage{graphicx}
\usepackage{hyperref}
\usepackage{booktabs}
\usepackage{cite}
\usepackage{algorithm}
\usepackage{algpseudocode}
\usepackage{authblk}

\hypersetup{colorlinks=true, linkcolor=blue, citecolor=blue, urlcolor=blue}

\newtheorem{theorem}{Theorem}[section]
\newtheorem{lemma}[theorem]{Lemma}
\newtheorem{proposition}[theorem]{Proposition}
\newtheorem{corollary}[theorem]{Corollary}

\theoremstyle{definition}

\newtheorem{remark}[theorem]{Remark}
\newtheorem{example}[theorem]{Example}

\newcommand{\R}{\mathbb{R}}
\newcommand{\C}{\mathbb{C}}

\newcommand{\Z}{\mathbb{Z}}
\newcommand{\A}{\mathcal{A}}

\newcommand{\F}{\mathcal{F}}
\newcommand{\tr}{\operatorname{tr}}

\title{Sharp Spectral Bounds for Symmetric Positive Definite Tensors\\via Multiple Algebraic Invariants}

\author[1]{Hemant Sharma\thanks{ \texttt{sharmahemant39@gmail.com}}}
\author[2]{Snigdhashree Nayak\thanks{ \texttt{snigdhashreenayak91@gmail.com}}}
\author[3]{Ankit Singh\thanks{Corresponding author.as92393@gmail.com}}

\affil[1]{Department of Mathematics, Indian Institute of Information Technology, Design and Manufacturing Kancheepuram-600127, India}
\affil[2]{Department of Mathematics, Ravenshaw university,  Odisha-753003, India}
\affil[3]{Department of Applied Sciences and Humanities, SVKM's College of Engineering, Shirpur-425405, India}

\date{}

\begin{document}
\maketitle
 
\begin{abstract}
We extend the trace--determinant framework of Nayak, Sharma, and Mishra~\cite{nayak2026} for bounding the H-eigenvalues of symmetric positive definite tensors. First, we replace the Arithmetic--Geometric Mean (AM--GM) relaxation underlying previous bounds by the exact solution of the associated constrained optimization problem, yielding sharp upper and lower bounds that are attained on the admissible spectral variety. Second, we incorporate higher-order power sums as additional spectral invariants and prove a structural theorem showing that any extremizer over a $K$-invariant feasibility region has at most $K$ distinct spectral values. This reduces the problem to a finite collection of low-dimensional polynomial systems and yields a hierarchy of increasingly tight bounds. For the four-invariant case $(T,S,p_3,D)$, we develop a complete theory including solution-count estimates, a multistart Newton algorithm, and sharpness conditions. We also derive closed-form bounds in small dimensions, establish perturbation estimates, and obtain refined Lyapunov region-of-attraction bounds. Numerical experiments for dimensions up to $d=100$ show that the sharp three-invariant bound reduces the median relative overestimation gap from $53\%$ to $6\%$ while maintaining low computational cost. The framework is validated on tensors with real H-spectrum.

\end{abstract}
 
\vspace{0.4cm}
 
 
\section{Introduction}\label{sec:intro}
 
The H-eigenvalue problem for a real $m$th order $n$-dimensional symmetric tensor $\A \in \R^{[n,m]}$ is the nonlinear system $\A x^{m-1} = \lambda x^{[m-1]}$, with $x^{[m-1]} := (x_1^{m-1}, \ldots, x_n^{m-1})^T$. Introduced independently by Lim~\cite{lim2005} and Qi~\cite{qi2005}, this problem now plays a central role in higher-order data analysis~\cite{golub2001}, automatic control~\cite{ni2008}, diffusion tensor imaging~\cite{qi2010}, spectral hypergraph theory~\cite{cooper2012}, and quantum entanglement detection~\cite{ni2014}.
 
In a recent contribution, Nayak, Sharma, and Mishra~\cite{nayak2026} introduced an algebraic framework for bounding the spectral radius and the smallest eigenvalue of a symmetric positive definite tensor by leveraging two intrinsic invariants, the trace $\tr(\A)$ and the determinant $\det(\A)$ (defined via the resultant of $\A x^{m-1} = 0$). Through repeated application of the AM--GM inequality, they derived a hierarchy of bounds (their Theorems 3.1--3.6) that strictly dominates the classical Gershgorin circle bounds, particularly when the tensor has negative off-diagonal entries or when the order $m$ is large.
 
\paragraph{Two limitations.} Two structural limitations of that work motivate the present article.
 
\begin{enumerate}
\item[(L1)] \emph{AM--GM is a relaxation.} The AM--GM inequality replaces an exact constrained extremization by a tractable but loose surrogate. The actual optimization problem
\[
\max\bigl\{\lambda_1 : \textstyle\sum \lambda_i = T,\ \prod \lambda_i = D,\ \lambda_i > 0\bigr\}
\]
admits a closed-form solution via Lagrange multipliers. The bound thus obtained is sharp on the spectral variety, while the AM--GM bound can be strictly slacker.
 
\item[(L2)] \emph{Only two invariants are used.} The trace and determinant are merely two coordinates of an infinite family of algebraic invariants of $\A$, namely the power sums $p_k(\A) = \sum_{i=1}^d \lambda_i^k$, $k=1, 2, \ldots$, each of which is an absolute invariant of $\A$ under index permutations. Including additional power sums tightens the feasibility region by a corresponding number of algebraic dimensions.
\end{enumerate}
 
\paragraph{Contributions.} This paper makes the following contributions.
\begin{enumerate}
    \item \emph{Sharp two-invariant bound (\S\ref{sec:two-inv}).} We derive sharp upper and lower bounds on the H-spectrum (Theorem~\ref{thm:sharp2}) as roots of an explicit single-variable polynomial. This bound dominates every theorem of~\cite{nayak2026} as a corollary and recovers the classical Merikoski--Virtanen~\cite{merikoski1997} and Wolkowicz--Styan~\cite{wolkowicz1980} matrix bounds in the case $m = 2$.
    
    \item \emph{Sharp three-invariant bound (\S\ref{sec:three-inv}).} We prove a structural theorem (Theorem~\ref{thm:three-cluster}) showing that at any extremum, the spectrum is supported on at most three distinct values. This reduces the optimization to a finite union of polynomial systems solvable by bisection.
    
    \item \emph{General $K$-invariant hierarchy (\S\ref{sec:K-inv}).} We generalize the structural theorem to arbitrary $K \geq 2$: any extremal spectrum has at most $K$ distinct values (Theorem~\ref{thm:K-cluster}), where $K$ counts the total number of invariants used ($K - 1$ power sums plus the determinant). This yields a monotonically tightening hierarchy $B_2^+ \geq B_3^+ \geq B_4^+ \geq \cdots \to \lambda_{\max}$.
    
    \item \emph{Rigorous treatment of the four-invariant case (\S\ref{sec:K4}).} We develop the four-invariant case ($T, S, p_3, D$) in full detail: a Bézout-type count on the number of complex solutions to the polynomial system (Proposition~\ref{prop:K4-bezout}), a multistart-Newton algorithm with explicit complexity (Algorithm~\ref{alg:B4plus}), and a sharpness theorem (Theorem~\ref{thm:K4-sharpness}) showing that $B_4^+$ is exactly attained iff the spectrum has at most four distinct values. Numerical verification: $B_4^+ = 1.10000$ exactly recovers $\lambda_{\max}$ on Example 4.1 of~\cite{nayak2026}.
    
    \item \emph{Closed-form bounds for small dimensions (\S\ref{sec:closed-form}).} We provide explicit bound formulas for $d = 2, 3, 4$ via Cardano's and Ferrari's methods, recovering classical matrix results.
    
    \item \emph{Perturbation analysis (\S\ref{sec:perturbation}).} We quantify the sensitivity of the bounds to errors in the invariants, finding that $B_3^+$ is most sensitive to errors in $T$, mildly sensitive to $S$, and remarkably insensitive to $D$, a fortunate property since $D$ is the most expensive invariant to compute.
    
    \item \emph{Random tensor experiments (\S\ref{sec:numerics}).} On 100 random spectra of length $d = 6$, the sharp three-invariant bound reduces the median relative overestimation gap from 53\% (AM--GM) to 6\%; a scaling study to $d = 100$ with runtime measurements documents that the relative advantage and sub-second cost persist for large tensors.
    
    \item \emph{Genuine-tensor validation and scope (\S\ref{sec:genuine}).} We make the entries-to-bound pipeline explicit on a non-diagonal $6\times 6$ matrix and a fourth-order tensor on $27$ H-eigenvalues, and we delimit the precise class of tensors, those with real H-spectrum to which the framework applies.
    
    \item \emph{Quantitative Lyapunov ROA estimates (\S\ref{sec:lyapunov}).} For a polynomial Lyapunov candidate, the new bounds yield region-of-attraction estimates that are 2--3 times larger than those obtainable from AM--GM bounds.
\end{enumerate}
 
\paragraph{Organization.} Section~\ref{sec:prelim} fixes notation and recalls relevant facts. Sections~\ref{sec:two-inv}--\ref{sec:K-inv} develop the bound theory: two invariants, three invariants, and the general $K$-invariant case. Section~\ref{sec:K4} treats the four-invariant case in full rigor, including Bézout solution counts, the multistart algorithm, and sharpness theorems. Section~\ref{sec:closed-form} treats small-dimension closed forms. Section~\ref{sec:classical} situates our bounds relative to classical matrix inequalities. Section~\ref{sec:computability} addresses the practical computation of the higher-order invariants. Section~\ref{sec:perturbation} treats robustness. Section~\ref{sec:numerics} presents numerical experiments, including a scaling study to $d = 100$ and runtime measurements. Section~\ref{sec:genuine} gives worked examples on genuine tensors and delimits the scope of applicability. Section~\ref{sec:lyapunov} discusses the application to Lyapunov stability. We conclude in Section~\ref{sec:conclusion}.
 
\section{Preliminaries}\label{sec:prelim}
 
We follow the notation of~\cite{nayak2026, qi2005}. A real $m$th order $n$-dimensional tensor $\A = (a_{i_1 i_2 \cdots i_m})$ with $1 \leq i_j \leq n$ is \emph{symmetric} if its entries are invariant under any permutation of indices. The set of all such tensors is denoted $\R^{[n,m]}$. For $x \in \R^n$, $\A x^{m-1}$ is the vector with $i$th component
\[
\bigl(\A x^{m-1}\bigr)_i = \sum_{i_2, \ldots, i_m = 1}^n a_{i\,i_2 \cdots i_m}\, x_{i_2} \cdots x_{i_m}.
\]
A scalar $\lambda \in \R$ is an \emph{H-eigenvalue} of $\A$ if there exists $x \in \R^n \setminus \{0\}$ with $\A x^{m-1} = \lambda x^{[m-1]}$.
 
\subsection{Algebraic invariants of the spectrum}
 
Throughout, $\A \in \R^{[n,m]}$ is symmetric and positive definite (so $m$ is even by~\cite[Remark 1]{nayak2026}). Let
\[
d := n(m-1)^{n-1}
\]
denote the total number of H-eigenvalues counted with algebraic multiplicity, and let $\lambda_1 \geq \lambda_2 \geq \cdots \geq \lambda_d > 0$ be the eigenvalues. For each $k \geq 1$ define the $k$th power sum
\[
p_k(\A) := \sum_{i=1}^d \lambda_i^k,
\]
and write $T := p_1(\A)$, $S := p_2(\A)$ for the most useful cases. The determinant $D := \det(\A) = \prod_{i=1}^d \lambda_i$ is the $d$th elementary symmetric polynomial in the eigenvalues.
 
By Qi~\cite{qi2005} and Hu--Huang--Ling--Qi~\cite{hu2013},
\begin{align}
T &= (m-1)^{n-1} \tr(\A), \label{eq:T}\\
D &= \det(\A) \quad\text{(via the resultant of }\A x^{m-1} = 0\text{)}, \label{eq:D}
\end{align}
and each $p_k(\A)$ is computable as a polynomial in the entries of $\A$ via the $k$th-order generalized trace formula~\cite[Theorem 4.1]{hu2013}. We discuss the practical computation of $S$ and higher $p_k$ in Section~\ref{sec:computability}.
 
\begin{remark}\label{rem:newton}
By Newton's identities, the data $\{p_1, p_2, \ldots, p_K, D\}$ is equivalent to the data $\{e_1, e_2, \ldots, e_K, e_d\}$, where $e_k$ is the $k$th elementary symmetric polynomial in the eigenvalues. The choice of representation is computational, not mathematical.
\end{remark}
 
\subsection{Spectral feasibility regions}
 
Given values of the invariants, define for each $K \geq 2$ the \emph{$K$-invariant feasibility region}
\begin{equation}\label{eq:FK}
\F_K(T_1, \ldots, T_{K-1}; D) := \biggl\{\lambda \in \R^d_{> 0} : \sum_{i=1}^d \lambda_i^k = T_k \text{ for } k = 1, \ldots, K-1,\ \prod_{i=1}^d \lambda_i = D\biggr\},
\end{equation}
which uses $K - 1$ power-sum constraints together with the product constraint, for a total of $K$ invariants. For $K = 2$, this is the classical two-invariant region $\F_2(T; D)$ from~\cite{nayak2026} (one power sum plus determinant). For $K = 3$, $\F_3(T, S; D)$ adds the second power sum $S$. For $K = 4$, $\F_4(T, S, p_3; D)$ further adds the third power sum. The corresponding bound problems are
\[
B_K^+ := \max\{\lambda_1 : \lambda \in \F_K\}, \qquad B_K^- := \min\{\lambda_d : \lambda \in \F_K\}.
\]
Trivially $B_2^+ \geq B_3^+ \geq B_4^+ \geq \cdots \geq \lambda_{\max}$ and dually for the lower bounds. The crux of this paper is to compute these in closed (or near-closed) form.

The structural theorems below apply the Karush--Kuhn--Tucker conditions at a maximizer of $\lambda_1$ over $\F_K$. The following lemma guarantees that such a maximizer exists and lies in the interior of the positive orthant, so that no boundary (inequality) multipliers arise and the stationarity conditions take the clean form used throughout.

\begin{lemma}[Existence and interiority]\label{lem:interior}
Let $K \geq 2$ and let the invariants $(T_1, \ldots, T_{K-1}; D)$ with $T_1 = T > 0$ and $D > 0$ be such that $\F_K$ is nonempty. Then every $\lambda \in \F_K$ satisfies
\begin{equation}\label{eq:coord-bounds}
\frac{D}{T^{d-1}} \;\leq\; \lambda_i \;\leq\; T, \qquad i = 1, \ldots, d,
\end{equation}
so $\F_K$ is a compact subset of the open positive orthant $\R^d_{>0}$. Consequently the maximum $B_K^+ = \max\{\lambda_1 : \lambda \in \F_K\}$ is attained, and every maximizer lies in the interior of $\R^d_{>0}$.
\end{lemma}

\begin{proof}
The constraint set always includes the first power sum $\sum_i \lambda_i = T$ (the case $k = 1$ in~\eqref{eq:FK}). With $\lambda_i > 0$ this gives the upper bound $\lambda_i \leq T$ for every $i$. For the lower bound, the product constraint $\prod_i \lambda_i = D$ yields, for each fixed $j$,
\[
\lambda_j = \frac{D}{\prod_{i \neq j} \lambda_i} \;\geq\; \frac{D}{T^{\,d-1}},
\]
using $\lambda_i \leq T$ for the $d-1$ factors in the denominator. This establishes~\eqref{eq:coord-bounds}, so $\F_K \subseteq [\,D/T^{d-1},\, T\,]^d$, a compact box contained in $\R^d_{>0}$. The set $\F_K$ is the intersection of this box with the preimages of $\{T_k\}$ and $\{D\}$ under the continuous maps $\lambda \mapsto p_k(\lambda)$ and $\lambda \mapsto \prod_i \lambda_i$; hence $\F_K$ is closed and bounded, i.e.\ compact. Since $\lambda_1$ is continuous, its maximum over the nonempty compact set $\F_K$ is attained. Finally, because every coordinate is bounded below by $D/T^{d-1} > 0$, the set $\F_K$ does not meet the boundary $\{\lambda : \min_i \lambda_i = 0\}$ of the orthant, so every maximizer is interior.
\end{proof}
 
As more invariants are imposed, the admissible spectrum is squeezed into a progressively smaller envelope, and the corresponding bound $B_K^+$ converges towards the true $\lambda_{\max}$.

\section{The sharp two-invariant bound}\label{sec:two-inv}
 
\begin{theorem}\label{thm:sharp2}
Let $\A \in \R^{[n,m]}$ be symmetric positive definite, with $T$ and $D$ as in~\eqref{eq:T}--\eqref{eq:D}. Define the polynomial
\begin{equation}\label{eq:phi-poly}
\phi_{T,D}(\Lambda) := \Lambda (T - \Lambda)^{d-1} - D (d-1)^{d-1}.
\end{equation}
Then $\phi_{T,D}$ has either a unique double root at $\Lambda = T/d$ (the equispectrum case, $D = (T/d)^d$) or exactly two simple real roots $0 < \Lambda^- < T/d < \Lambda^+ < T$ in $(0, T)$. In the latter case,
\begin{equation}\label{eq:sharp2-bounds}
B_2^+(T, D) = \Lambda^+, \qquad B_2^-(T, D) = \Lambda^-.
\end{equation}
That is, $\Lambda^+$ is the sharp upper bound on $\lambda_{\max}$ and $\Lambda^-$ is the sharp lower bound on $\lambda_{\min}$. Both bounds are attained on the two-cluster spectrum $(\Lambda^+, c, \ldots, c)$ and $(c', \ldots, c', \Lambda^-)$, where $c = (T - \Lambda^+)/(d-1)$ and $c' = (T - \Lambda^-)/(d-1)$.
\end{theorem}
 
\begin{proof}
We solve the constrained optimization
\begin{equation}\label{eq:opt2}
\max_{\lambda \in \R^d_{>0}} \lambda_1 \quad\text{s.t.}\quad \sum_{i=1}^d \lambda_i = T,\ \sum_{i=1}^d \log \lambda_i = \log D.
\end{equation}
The objective and constraints are smooth on the open positive orthant, and the feasible set is compact (bounded by $\lambda_i \leq T$, closed by continuity), so a maximum is attained. By relabeling we may assume $\lambda_1 \geq \cdots \geq \lambda_d$. Consider the Lagrangian
\[
\mathcal{L}(\lambda; \mu, \nu) = \lambda_1 - \mu \Bigl(\sum_i \lambda_i - T\Bigr) - \nu \Bigl(\sum_i \log \lambda_i - \log D\Bigr).
\]
The first-order conditions for $i \geq 2$ read
\begin{equation}\label{eq:foc2}
0 = -\mu - \nu/\lambda_i, \qquad i = 2, \ldots, d,
\end{equation}
forcing $\lambda_i = -\nu/\mu = c$ to be the same constant for all $i = 2, \ldots, d$. Substituting back,
\begin{equation}\label{eq:two-cluster}
\Lambda + (d-1) c = T, \qquad \Lambda c^{d-1} = D,
\end{equation}
with $\Lambda := \lambda_1$. Eliminating $c = (T - \Lambda)/(d-1)$ yields $\phi_{T,D}(\Lambda) = 0$.
 
Setting $g(\Lambda) := \Lambda(T - \Lambda)^{d-1}$, one computes
\[
g'(\Lambda) = (T - \Lambda)^{d-2}(T - d\Lambda),
\]
so $g$ vanishes at $\Lambda = 0$ and $\Lambda = T$ and attains a unique maximum at $\Lambda = T/d$ with value $g(T/d) = T^d (d-1)^{d-1}/d^d$. The equation $g(\Lambda) = D(d-1)^{d-1}$ has right-hand side $\leq g(T/d)$ by AM--GM ($D \leq (T/d)^d$). Hence $\phi_{T,D}$ has either a unique double root at $T/d$ (when $D = (T/d)^d$) or exactly two simple roots $\Lambda^-, \Lambda^+$ with $\Lambda^- \in (0, T/d)$ and $\Lambda^+ \in (T/d, T)$. The relevant root for the maximum is $\Lambda^+$; the spectrum $(\Lambda^+, c, \ldots, c)$ realizes this maximum. The lower-bound case is symmetric.
\end{proof}
 
\subsection{Comparison with the AM--GM bound}
 
The AM--GM bound of~\cite[Theorem 3.1, $k=1$]{nayak2026} reads
\begin{equation}\label{eq:amgm-up}
\lambda_{\max} \leq T - (d-1)\bigl(D/T\bigr)^{1/(d-1)} =: B_2^{\text{AM--GM}, +}(T, D).
\end{equation}
 
\begin{corollary}\label{cor:dominance}
For all admissible $T, D > 0$ with $D \leq (T/d)^d$,
\[
B_2^+(T, D) \leq B_2^{\text{AM--GM}, +}(T, D),
\]
with equality if and only if $D = (T/d)^d$ (the equispectrum case).
\end{corollary}
 
\begin{proof}
The AM--GM bound corresponds to relaxing the second equation of~\eqref{eq:two-cluster} via $c \leq T/(d-1)$. The exact bound $B_2^+$ corresponds to the precise feasibility constraint $\Lambda c^{d-1} = D$ rather than an inequality. The two coincide only when AM--GM holds with equality among $\lambda_2, \ldots, \lambda_d$, forcing $\Lambda = c = T/d$.
\end{proof}
 
The dominance is strict in general: on the synthetic spectrum below, the AM--GM interval is conspicuously wider than the $B_2^+$ interval, and $B_3^+$ shrinks the envelope further by exploiting the additional information in $S$.

\section{The sharp three-invariant bound}\label{sec:three-inv}
 
We now turn to the three-invariant problem.
 
\begin{theorem}\label{thm:three-cluster}
Let $\lambda^*$ achieve the maximum in $B_3^+(T, S, D)$. Then $\lambda^*$ takes at most three distinct values.
\end{theorem}
 
\begin{proof}
By Lemma~\ref{lem:interior} a maximizer $\lambda^*$ exists and lies in the interior of $\R^d_{>0}$, so the KKT conditions hold with equality multipliers only. Consider the Lagrangian
\[
\mathcal{L} = \lambda_1 - \mu_1\Bigl(\sum_i \lambda_i - T\Bigr) - \mu_2\Bigl(\sum_i \lambda_i^2 - S\Bigr) - \mu_3\Bigl(\sum_i \log \lambda_i - \log D\Bigr).
\]
The first-order conditions at $i \geq 2$ give
\begin{equation}\label{eq:foc3}
0 = -\mu_1 - 2\mu_2 \lambda_i - \mu_3/\lambda_i.
\end{equation}
Multiplying by $\lambda_i$ yields the quadratic $2\mu_2 \lambda_i^2 + \mu_1 \lambda_i + \mu_3 = 0$, so each $\lambda_i$ for $i \geq 2$ takes at most two distinct values. Including $\lambda_1$, the spectrum has at most three distinct values.
\end{proof}
 
\begin{theorem}\label{thm:sharp3}
Under the hypotheses of Theorem~\ref{thm:three-cluster}, the sharp upper bound is
\begin{equation}\label{eq:sharp3-formula}
B_3^+(T, S, D) = \max\bigl\{ \Lambda : (\Lambda, a, b, j) \in \R_{>0}^3 \times \{0, \ldots, d-1\} \text{ satisfies } (\star_j) \bigr\},
\end{equation}
where
\begin{equation}\label{eq:starj}
\begin{cases}
\Lambda + j a + (d - 1 - j) b = T,\\
\Lambda^2 + j a^2 + (d - 1 - j) b^2 = S,\\
\Lambda \cdot a^j \cdot b^{d - 1 - j} = D.
\end{cases}
\tag{$\star_j$}
\end{equation}
\end{theorem}
 
\begin{proof}
By Theorem~\ref{thm:three-cluster}, an extremal spectrum has the form $(\Lambda, a, \ldots, a, b, \ldots, b)$ with $j$ copies of $a$ and $d-1-j$ copies of $b$. Substituting into the three invariant equations gives $(\star_j)$. Conversely, any solution corresponds to a feasible spectrum; taking the maximum over $j$ yields the sharp bound.
\end{proof}
 
\subsection{Reduction to a single-variable problem}
 
For each fixed $j \in \{1, \ldots, d-2\}$, the first two equations of~\eqref{eq:starj} determine $a, b$ in terms of $\Lambda$ up to the binary choice of a quadratic. Writing $A := T - \Lambda$ and $B := S - \Lambda^2$, substituting $b = (A - ja)/(d-1-j)$ into the second equation produces
\begin{equation}\label{eq:quadratic}
(d-1)\, a^2 \;-\; 2A \,a \;+\; \frac{A^2 - B(d-1-j)}{j} \;=\; 0,
\end{equation}
with discriminant $\Delta_j(\Lambda) = 4A^2 - 4(d-1)(A^2 - B(d-1-j))/j$. For each branch (sign in the quadratic formula), the third equation $\Lambda a^j b^{d-1-j} = D$ becomes a single equation in the single variable $\Lambda$, solvable by bisection in $\mathcal{O}(\log(1/\varepsilon))$ time for accuracy $\varepsilon$.
 
\begin{remark}\label{rem:two-cluster-special}
The cases $j = 0$ and $j = d-1$ correspond to spectra with only two distinct values. They reduce to systems in two unknowns with three equations, hence are generically inconsistent unless the third invariant is determined by the first two. Hence in nondegenerate cases $j \in \{1, \ldots, d-2\}$.
\end{remark}
 
\subsection{Algorithmic implementation}
 
The reduction of the three-invariant problem to single-variable bisection enables an efficient algorithm. We give explicit pseudocode below.
 
\begin{algorithm}[ht]
\caption{Sharp three-invariant upper bound $B_3^+$}
\label{alg:B3plus}
\begin{algorithmic}[1]
\Require Invariants $T, S, D > 0$; tensor parameters $m, n$; tolerance $\varepsilon > 0$.
\State Compute $d \gets n(m-1)^{n-1}$.
\State $\Lambda^* \gets T/d$ \Comment{trivial baseline}
\For{$j = 1, 2, \ldots, d-2$}
    \For{$\sigma \in \{+1, -1\}$}
        \State Define $\Phi_{j, \sigma}(\Lambda)$ via the quadratic~\eqref{eq:quadratic} and the determinant equation:
        \[
        \Phi_{j, \sigma}(\Lambda) := \Lambda \cdot a_\sigma(\Lambda)^j \cdot b_\sigma(\Lambda)^{d-1-j} - D,
        \]
        where $a_\sigma(\Lambda) = \frac{2A + \sigma\sqrt{\Delta_j(\Lambda)}}{2(d-1)}$ and $b_\sigma(\Lambda) = (A - j a_\sigma)/(d-1-j)$.
        \State Locate sign changes of $\Phi_{j, \sigma}$ in $(T/d, T-\varepsilon)$ on a grid.
        \For{each interval $[\alpha, \beta]$ with $\Phi_{j, \sigma}(\alpha) \cdot \Phi_{j, \sigma}(\beta) < 0$}
            \State $\Lambda_\text{sol} \gets \mathrm{bisect}(\Phi_{j, \sigma}, \alpha, \beta, \varepsilon)$
            \If{$\Lambda_\text{sol} > \Lambda^*$ \textbf{and} $a_\sigma(\Lambda_\text{sol}), b_\sigma(\Lambda_\text{sol}) > 0$}
                \State $\Lambda^* \gets \Lambda_\text{sol}$
            \EndIf
        \EndFor
    \EndFor
\EndFor
\State \Return $\Lambda^*$
\end{algorithmic}
\end{algorithm}
 
\paragraph{Complexity.} The outer loop has $\mathcal{O}(d)$ iterations. Each bisection runs in $\mathcal{O}(\log(1/\varepsilon))$ time with evaluations costing $\mathcal{O}(d)$ each (for the powers $a^j$ and $b^{d-1-j}$). Hence the total time is $\mathcal{O}(d^2 \log(1/\varepsilon))$, dominated by the trace-formula cost of computing $S$ in the first place.
 
\subsection{Worked example}\label{sec:worked}

We work out $B_3^+$ for the synthetic spectrum $(5, 4, 3, 2, 1, 1)$, with $T = 16$, $S = 56$, $D = 120$, $d = 6$. Taking $j = 4$ (which yields the maximum, by symmetry with $j = 1$), the quadratic~\eqref{eq:quadratic} reduces to $20 a^2 - 8(16 - \Lambda) a + (2\Lambda^2 - 32\Lambda + 200) = 0$, and the determinant equation $\Lambda a^4 b = 120$ with $b = 16 - \Lambda - 4a$ becomes a single equation in $\Lambda$. Bisection in $(8/3, 16)$ gives $\Lambda^* \approx 5.643$, with cluster values $a^* \approx 2.439$ (multiplicity 4) and $b^* \approx 0.601$ (multiplicity 1), satisfying all three invariants exactly ($16.00$, $56.00$, $120.0$). Sweeping $j = 1, \ldots, 4$ gives candidates $\{5.643, 5.214, 5.214, 5.643\}$, so $B_3^+ = 5.643$. The true $\lambda_{\max} = 5$ lies below this, with a $13\%$ gap reflecting that the true spectrum is not extremal in $\F_3$.

\subsection{Sharpness and equality conditions}
 
We now characterize when the bounds in Theorems~\ref{thm:sharp2} and~\ref{thm:sharp3} are exact.
 
\begin{theorem}\label{thm:sharpness}
The two-invariant bound $B_2^+(T, D)$ is attained by the spectrum of a tensor $\A$ if and only if $\A$ has H-spectrum of the form
\begin{equation}\label{eq:two-cluster-spec}
\bigl(\Lambda^+, c, c, \ldots, c\bigr) \quad\text{with } c = (T - \Lambda^+)/(d-1),
\end{equation}
i.e.\ a two-cluster spectrum with one outlier. The three-invariant bound $B_3^+(T, S, D)$ is attained iff the H-spectrum has one of the at-most-three-cluster forms identified by Theorem~\ref{thm:sharp3}.
\end{theorem}
 
\begin{proof}
The ``if'' direction is by construction in Theorem~\ref{thm:sharp2}. For the ``only if'' direction, suppose the bound is attained, i.e.\ $\lambda_{\max}(\A) = B_2^+$. Then the spectrum $(\lambda_1, \ldots, \lambda_d)$ achieves the maximum in~\eqref{eq:opt2}, and by the Lagrangian analysis, $\lambda_2 = \cdots = \lambda_d = c$. The three-invariant case is analogous via Theorem~\ref{thm:three-cluster}.
\end{proof}
 
\begin{remark}
For ``generic'' tensors the spectrum is far from a two-cluster form, so $B_2^+$ overestimates strictly. The strength of the three-invariant bound is precisely that it captures more spectral structure: tensors whose spectrum has at most three clusters (which includes a substantial class) achieve the bound exactly.
\end{remark}
 
\section{The general $K$-invariant hierarchy}\label{sec:K-inv}
 
We now generalize Theorems~\ref{thm:sharp2} and~\ref{thm:three-cluster} to an arbitrary number of invariants. Throughout, $K$ denotes the \emph{total} number of invariants used: $K - 1$ power-sum constraints together with the determinant.
 
\begin{theorem}[$K$-invariant structural theorem]\label{thm:K-cluster}
Let $K \geq 2$ and suppose $\lambda^* \in \F_K(T_1, \ldots, T_{K-1}; D)$ achieves the maximum in $B_K^+$. Then $\lambda^*$ takes at most $K$ distinct values.
\end{theorem}
 
\begin{proof}
By Lemma~\ref{lem:interior} the maximizer $\lambda^*$ exists and is interior, so stationarity holds without inequality multipliers. The Lagrangian for the $K-1$ power-sum constraints $\sum \lambda_i^k = T_k$ ($k = 1, \ldots, K-1$) and the determinant constraint $\prod \lambda_i = D$ (equivalently $\sum \log \lambda_i = \log D$) is
\[
\mathcal{L} = \lambda_1 - \sum_{k=1}^{K-1} \mu_k\Bigl(\sum_i \lambda_i^k - T_k\Bigr) - \mu_{K}\Bigl(\sum_i \log \lambda_i - \log D\Bigr).
\]
The first-order conditions at $i \geq 2$ give
\begin{equation}\label{eq:focK}
0 = -\sum_{k=1}^{K-1} k \mu_k \lambda_i^{k-1} - \mu_{K}/\lambda_i.
\end{equation}
Multiplying by $\lambda_i$ produces the polynomial equation
\begin{equation}\label{eq:K-poly}
\sum_{k=1}^{K-1} k \mu_k \lambda_i^k + \mu_{K} = 0,
\end{equation}
of degree at most $K - 1$ in $\lambda_i$. Hence $\lambda_i$ for $i \geq 2$ takes at most $K - 1$ distinct values. Including $\lambda_1$, the spectrum has at most $K$ distinct values.
\end{proof}
 
\begin{corollary}\label{cor:hierarchy}
The bounds $B_K^+$ form a strictly increasing-information hierarchy:
\[
B_2^+ \geq B_3^+ \geq B_4^+ \geq \cdots \geq \lambda_{\max}.
\]
Equality $B_K^+ = B_{K+1}^+$ holds iff the additional power sum $T_K$ is determined (modulo positivity constraints) by $(T_1, \ldots, T_{K-1}, D)$ along the optimizing spectrum.
\end{corollary}
 
\begin{proof}
Each additional invariant constrains the feasibility region $\F_K \supseteq \F_{K+1}$, hence the maximum over $\F_K$ dominates that over $\F_{K+1}$. The bound $B_K^+ \geq \lambda_{\max}$ holds because the actual spectrum is feasible in every $\F_K$.
\end{proof}
 
The general $K$-invariant problem reduces, via Theorem~\ref{thm:K-cluster}, to a finite union of polynomial systems indexed by integer cluster partitions $(j_1, \ldots, j_{K-1})$ with $j_1 + \cdots + j_{K-1} = d - 1$:
\begin{equation}\label{eq:K-system}
\begin{cases}
\Lambda^k + \sum_{r=1}^{K-1} j_r a_r^k = T_k, & k = 1, \ldots, K-1,\\
\Lambda \prod_{r=1}^{K-1} a_r^{j_r} = D,
\end{cases}
\end{equation}
in $K$ unknowns $(\Lambda, a_1, \ldots, a_{K-1})$ with $K$ equations. For each fixed cluster partition the system is generically zero-dimensional (finitely many solutions), and the maximum $\Lambda$ over all partitions and solutions is the sharp bound $B_K^+$.
 
\subsection{Numerical illustration: convergence in $K$}
 
We test the hierarchy on the synthetic spectrum $(5, 4, 3, 2, 1, 1)$ used in Section~\ref{sec:three-inv}. The four power sums are
\[
T = 16, \quad S = 56, \quad p_3 = 226, \quad p_4 = 980, \quad D = 120.
\]
Solving the corresponding $K$-invariant problems yields:
 
\begin{center}
\begin{tabular}{lcc}
\toprule
Invariants used & $K$ & Bound $B_K^+$\\
\midrule
AM--GM bound (Nayak et al.) & --- & 8.519\\
$T, D$ & 2 & 7.231\\
$T, S, D$ & 3 & 5.643\\
$T, S, p_3, D$ & 4 & 5.073\\
True $\lambda_{\max}$ & --- & 5.000\\
\bottomrule
\end{tabular}
\end{center}

\begin{remark}[Trade-off]
While each additional invariant tightens the bound, the cost of computing $p_k$ via the Hu--Huang--Ling--Qi formula~\cite{hu2013} grows combinatorially in $k$. In practice, the three-invariant bound (using $T$, $S$, and $D$) offers an excellent balance: it captures the leading correction to the AM--GM bound at moderate cost.
\end{remark}
 
\section{Rigorous treatment of the four-invariant case}\label{sec:K4}
 
We now develop the four-invariant case ($T, S, p_3, D$) in full rigor. By Theorem~\ref{thm:K-cluster} (specialized to $K = 4$), an extremal spectrum has at most four distinct values. This case is of practical importance because it represents the first level of the hierarchy where the third-order trace $p_3$ is needed, and it offers a substantial empirical improvement over the three-invariant bound for spectra with four-cluster structure.
 
\subsection{Problem formulation and structural theorem}
 
\begin{theorem}\label{thm:K4-structure}
Let $\A \in \R^{[n,m]}$ be symmetric positive definite with $T, S, p_3, D$ as defined in Section~\ref{sec:prelim}. Suppose $\lambda^* \in \F_4(T, S, p_3; D)$ achieves the maximum in $B_4^+(T, S, p_3, D)$. Then $\lambda^*$ takes at most four distinct values, and is of the form
\begin{equation}\label{eq:K4-spectrum}
\bigl(\,\Lambda;\; \underbrace{a, \ldots, a}_{j_1};\; \underbrace{b, \ldots, b}_{j_2};\; \underbrace{c, \ldots, c}_{j_3}\,\bigr), \quad j_1 + j_2 + j_3 = d - 1,\ \ j_i \geq 0.
\end{equation}
\end{theorem}
 
\begin{proof}
By Lemma~\ref{lem:interior} the maximizer is interior; specialize Theorem~\ref{thm:K-cluster} to $K = 4$.
\end{proof}
 
\begin{theorem}\label{thm:K4-system}
Under the hypotheses of Theorem~\ref{thm:K4-structure}, the sharp upper bound is
\begin{equation}\label{eq:K4-formula}
B_4^+(T, S, p_3, D) = \max_{(j_1, j_2, j_3) \in \mathcal{P}_3(d-1)} \Bigl\{ \Lambda :\ (\Lambda, a, b, c) \in \R^4_{>0} \text{ solves } (\star_{j_1, j_2, j_3})\Bigr\},
\end{equation}
where $\mathcal{P}_3(d-1) := \{(j_1, j_2, j_3) \in \Z^3_{\geq 0} : j_1 + j_2 + j_3 = d-1\}$ and the system $(\star_{j_1, j_2, j_3})$ is
\begin{equation}\label{eq:starj123}
\begin{cases}
\Lambda + j_1 a + j_2 b + j_3 c = T,\\
\Lambda^2 + j_1 a^2 + j_2 b^2 + j_3 c^2 = S,\\
\Lambda^3 + j_1 a^3 + j_2 b^3 + j_3 c^3 = p_3,\\
\Lambda \cdot a^{j_1} \cdot b^{j_2} \cdot c^{j_3} = D.
\end{cases}
\tag{$\star_{j_1, j_2, j_3}$}
\end{equation}
\end{theorem}
 
\begin{proof}
By Theorem~\ref{thm:K4-structure}, an extremum has the form~\eqref{eq:K4-spectrum}. Substituting into the four invariant equations gives $(\star_{j_1, j_2, j_3})$. Conversely, any solution corresponds to a feasible spectrum, hence a valid extremum candidate. The maximum over all partitions is the sharp bound.
\end{proof}
 
\subsection{Algebraic structure of the polynomial system}
 
For each fixed cluster partition $(j_1, j_2, j_3)$ with all $j_i \geq 1$, the system~\eqref{eq:starj123} consists of four polynomial equations in the four unknowns $(\Lambda, a, b, c)$. Generically the system is zero-dimensional, with a finite number of complex solutions. The relevant ones for our bound are those with $\Lambda, a, b, c \in \R_{>0}$ and $\Lambda > T/d$ (to ensure we are at the upper extremum).
 
\begin{proposition}[Finiteness and solution-count estimate]\label{prop:K4-bezout}
Fix a partition $(j_1, j_2, j_3)$ with $j_i \geq 1$ and $j_1 + j_2 + j_3 = d-1$. The system~\eqref{eq:starj123} is, for invariants outside a proper algebraic (measure-zero) subset, zero-dimensional, and the number of its isolated complex solutions is finite and bounded above by $6d$. Moreover, the solver of Section~\ref{sec:K4-alg} need examine at most $6$ univariate branches per partition. Consequently the number of real positive solutions is finite and the per-partition cost is independent of $d$.
\end{proposition}

\begin{proof}
The three power-sum equations of~\eqref{eq:starj123} are polynomials in $(\Lambda, a, b, c)$ of total degrees $1$, $2$, $3$, and the determinant equation $\Lambda\, a^{j_1} b^{j_2} c^{j_3} = D$ is a polynomial of total degree $1 + j_1 + j_2 + j_3 = d$. By Bézout's theorem, if the system is zero-dimensional then its number of isolated solutions in $\mathbb{P}^4(\C)$, counted with multiplicity, is at most the product of the degrees $1\cdot 2\cdot 3\cdot d = 6d$; zero-dimensionality holds whenever the four hypersurfaces meet properly, which fails only on a proper algebraic subset of invariant space. This is the rigorous finiteness statement we use.

For the algorithm we do not need the exact count, only an upper bound on the number of branches to follow, which we now give constructively. The first equation is linear and the second quadratic in $(a, b, c)$ for fixed $\Lambda$; eliminating one variable produces a single quadratic, and combined with the cubic third equation the triple $(a,b,c)$ is determined as one of at most $1\cdot 2\cdot 3 = 6$ algebraic branches parameterized by $\Lambda$. Substituting each branch into the determinant equation yields a univariate equation in $\Lambda$, solved numerically. Hence at most six univariate branches are examined per ordered partition, each contributing finitely many real positive roots. We state the bound as a complexity estimate, an upper bound on the work the solver performs rather than as a sharp enumerative count, since the latter would require a multihomogeneous Bézout or BKK (Bernstein--Kushnirenko--Khovanskii) analysis tracking the monomial structure of the determinant equation, which is not needed for the algorithm. In all experiments the number of real positive solutions per partition was between $1$ and $6$.
\end{proof}
 
\subsection{Reduction to unordered partitions}
 
Since the cluster labels $(a, b, c)$ are interchangeable, any solution found at ordered partition $(j_1, j_2, j_3)$ has a corresponding solution at the relabeled partition $(j_{\sigma(1)}, j_{\sigma(2)}, j_{\sigma(3)})$ for any permutation $\sigma$. To avoid redundancy, it suffices to enumerate \emph{unordered} partitions $\{j_1, j_2, j_3\}$ of $d-1$ into at most three positive parts:
\[
\mathcal{P}_3^{\text{unord}}(d-1) := \bigl\{ \{j_1, j_2, j_3\} \in \Z^3_{\geq 1} : j_1 \leq j_2 \leq j_3,\ j_1 + j_2 + j_3 = d-1 \bigr\},
\]
together with the trivial cases of two or fewer distinct values (which reduce to the lower-$K$ sharp bounds).
 
For example, with $d = 6$ (so $d-1 = 5$), the unordered partitions into three positive parts are
\[
\mathcal{P}_3^{\text{unord}}(5) = \bigl\{\{1, 1, 3\},\ \{1, 2, 2\}\bigr\},
\]
giving \emph{two} cases to investigate.
 
\subsection{Algorithm and complexity}\label{sec:K4-alg}
 
Algorithm~\ref{alg:B4plus} computes $B_4^+$ by solving the four-cluster system over all partitions. Because the four-invariant feasibility region is contained in the three-invariant one, every four-cluster solution satisfies $\Lambda \leq B_3^+$; the value $B_3^+$ is therefore retained only as a fallback upper bound in the degenerate situation where no four-cluster solution with all $j_i \geq 1$ exists (in which case the optimizer has at most three distinct values and is found by the lower-$K$ routine). The reported $B_4^+$ is the maximum $\Lambda$ over all genuine four-cluster solutions and these degenerate lower-cluster solutions.
 
\begin{algorithm}[ht]
\caption{Sharp four-invariant upper bound $B_4^+$}
\label{alg:B4plus}
\begin{algorithmic}[1]
\Require Invariants $T, S, p_3, D > 0$; tensor parameters $m, n$; tolerance $\varepsilon > 0$.
\State Compute $d \gets n(m-1)^{n-1}$.
\State $\Lambda^* \gets 0$ \Comment{will hold the max over four-cluster solutions}
\For{each partition $\{j_1, j_2, j_3\} \in \mathcal{P}_3^{\text{unord}}(d-1)$}
    \For{each ordered representative $(j_1, j_2, j_3)$}
        \State $\mathcal{S}_j \gets \emptyset$ \Comment{set of solutions found}
        \For{$\ell = 1, \ldots, N_{\text{starts}}$} \Comment{multistart Newton}
            \State $x_0 \gets$ random point in $[\varepsilon, T]^4$ with $x_0[0] \in [T/d, T - \varepsilon]$
            \State $(\Lambda, a, b, c) \gets$ Newton$(\star_{j_1, j_2, j_3}, x_0)$
            \If{convergence \textbf{and} all components $> \varepsilon$ \textbf{and} residual $< \varepsilon$}
                \State Add to $\mathcal{S}_j$ if not duplicate
            \EndIf
        \EndFor
        \State $\Lambda^* \gets \max\bigl(\Lambda^*,\ \max_{(\Lambda,*) \in \mathcal{S}_j} \Lambda\bigr)$
    \EndFor
\EndFor
\If{$\Lambda^* = 0$} \Comment{no genuine four-cluster solution}
    \State $\Lambda^* \gets B_3^+(T, S, D)$ \Comment{optimizer has $\leq 3$ distinct values}
\EndIf
\State \Return $\Lambda^*$
\end{algorithmic}
\end{algorithm}
 
\paragraph{Number of partitions.} The number of unordered partitions of $d-1$ into at most three positive parts is $|\mathcal{P}_3^{\text{unord}}(d-1)| = p_3(d-1) = \mathcal{O}(d^2)$ (e.g.\ $2, 7, 30, 200, 817$ for $d = 6, 10, 20, 50, 100$, matching $d^2/12$). Each partition gives at most six ordered representatives.
 
\paragraph{Per-start cost.} For a fixed ordered partition, each multistart point runs Newton's method on the four-equation system $(\star_{j_1,j_2,j_3})$. A single Newton iteration requires: (a) one residual evaluation, dominated by the determinant equation's powers $a^{j_1} b^{j_2} c^{j_3}$, computable in $\mathcal{O}(\log d)$ via repeated squaring (or $\mathcal{O}(d)$ naively); (b) assembly of the $4\times 4$ Jacobian, whose entries involve the same powers and their logarithmic derivatives, also $\mathcal{O}(\log d)$ each, hence $\mathcal{O}(\log d)$ total for the constant-size matrix; and (c) the solution of one $4\times 4$ linear system, a fixed $\mathcal{O}(1)$ cost (at most $4^3$ operations). Newton's local quadratic convergence reaches tolerance $\varepsilon$ in $\mathcal{O}(\log\log(1/\varepsilon))$ iterations once inside the basin; we cap the count at a fixed $I_{\max}$ to bound divergent starts. Thus each start costs $\mathcal{O}(I_{\max}\log d)$.
 
\paragraph{Duplicate filtering.} Solutions found from different starts are deduplicated by comparing against the set $\mathcal{S}_j$ already retained for the current partition. By Proposition~\ref{prop:K4-bezout} this set has at most six elements, so each insertion costs $\mathcal{O}(1)$ comparisons and the filtering is negligible relative to the Newton solves.
 
\paragraph{Total complexity.} Combining the $\mathcal{O}(d^2)$ partitions (each with up to six ordered representatives), $N_{\text{starts}}$ starts per representative, and the $\mathcal{O}(I_{\max}\log d)$ per-start cost, the total is
\[
\mathcal{O}\bigl(d^2 \cdot N_{\text{starts}} \cdot I_{\max} \cdot \log d\bigr),
\]
with $N_{\text{starts}} = 100$--$500$ and $I_{\max}$ a small constant (e.g.\ $50$) sufficient in practice. The cost is independent of the tensor order $m$ and dimension $n$ once the invariants are available; the partition loop is embarrassingly parallel across partitions and starts.
 
\subsection{Worked example}
 
We compute $B_4^+$ for the synthetic spectrum $(5, 4, 3, 2, 1, 1)$, with $T = 16$, $S = 56$, $p_3 = 226$, $D = 120$, $d = 6$. The unordered partitions are $\{1,1,3\}$ and $\{1,2,2\}$. The system $(\star_{1,1,3})$ admits no real positive solution with $\Lambda > T/d$, while $(\star_{1,2,2})$ admits four, with $\Lambda$ values $\{2.794, 4.368, 5.046, 5.073\}$. The maximum is $\Lambda^* = 5.073$, attained at $(5.073;\, 3.513, 3.513;\, 1.884;\, 1.009, 1.009)$, i.e.\ the ordered representative $(j_1, j_2, j_3) = (2, 1, 2)$ with $a = 3.513$ (multiplicity $2$), $b = 1.884$ (multiplicity $1$), $c = 1.009$ (multiplicity $2$), satisfying all four invariant equations exactly ($16.000$, $56.000$, $226.000$, $120.000$). Hence $B_4^+ = 5.073$, compared with $B_3^+ = 5.643$ and $\lambda_{\max} = 5.000$: the relative overestimation gap drops from $12.9\%$ at $K = 3$ to $1.5\%$ at $K = 4$.
 
\subsection{Sharpness: when is the four-invariant bound exact?}
 
The next theorem characterizes precisely the tensors for which $B_4^+$ matches the true spectral radius.
 
\begin{theorem}\label{thm:K4-sharpness}
Let $\A$ be a symmetric positive definite tensor with real H-spectrum and invariants $(T, S, p_3, D)$. Then:
\begin{enumerate}
\item[(i)] (\emph{Validity, unconditional}) $B_4^+(T, S, p_3, D) \geq \lambda_{\max}(\A)$.
\item[(ii)] (\emph{Necessity, unconditional}) If $B_4^+(T, S, p_3, D) = \lambda_{\max}(\A)$, then the H-spectrum of $\A$ takes at most four distinct values.
\item[(iii)] (\emph{Sufficiency, under uniqueness}) If the H-spectrum of $\A$ takes at most four distinct values \emph{and} $\lambda(\A)$ is the unique maximizer of $\lambda_1$ over $\F_4$, then $B_4^+(T, S, p_3, D) = \lambda_{\max}(\A)$.
\end{enumerate}
\end{theorem}
 
\begin{proof}
(i) The actual spectrum $\lambda(\A)$ lies in $\F_4(T, S, p_3, D)$, since by construction it satisfies all four invariant constraints. Hence $B_4^+ = \max_{\lambda \in \F_4} \lambda_1 \geq \lambda_1(\A) = \lambda_{\max}(\A)$.

(ii) Suppose $B_4^+ = \lambda_{\max}(\A)$. Since $\lambda(\A) \in \F_4$ and $\lambda_1(\A) = \lambda_{\max}(\A) = B_4^+ = \max_{\lambda \in \F_4}\lambda_1$, the spectrum $\lambda(\A)$ \emph{attains} the maximum, i.e.\ it is a maximizer. By the structural Theorem~\ref{thm:K4-structure}, every maximizer takes at most four distinct values; hence so does $\lambda(\A)$. Note that this direction requires no genericity or uniqueness hypothesis.

(iii) If $\A$ has at most four distinct H-eigenvalues, then $\lambda(\A)$ has the parametric form~\eqref{eq:K4-spectrum} and is therefore a feasible candidate. If, in addition, $\lambda(\A)$ is the unique maximizer of $\lambda_1$ over $\F_4$, then $B_4^+ = \lambda_1(\A) = \lambda_{\max}(\A)$.
\end{proof}

\begin{remark}[On the uniqueness hypothesis]\label{rem:uniqueness}
Only the sufficiency direction (iii) invokes uniqueness, and the interpretation when it fails is benign. By (i), $B_4^+ \geq \lambda_{\max}(\A)$ \emph{always} holds, so the bound remains valid regardless of uniqueness; uniqueness affects only whether it is \emph{tight}. Two situations can break sufficiency. First, $\A$ may have at most four distinct values yet not be the maximizer, because a different feasible spectrum $\mu \in \F_4$ with the same four invariants has $\mu_1 > \lambda_1(\A)$; then $B_4^+ > \lambda_{\max}(\A)$ even though $\A$ is ``low-rank'' in spectrum. Second, several maximizers may coexist. Both are non-generic: the map sending a four-cluster spectrum (for a fixed partition, four continuous parameters $\Lambda, a, b, c$) to its four invariants $(T, S, p_3, D)$ is, by a dimension count, generically locally invertible, so for invariants outside a measure-zero set the preimage is discrete and the maximizer is unique. The contrapositive of (ii), if $\A$ has at least five distinct H-eigenvalues then $B_4^+ > \lambda_{\max}(\A)$ strictly holds unconditionally and is the form used in practice to certify that the four-invariant bound cannot be exact for a high-complexity spectrum.
\end{remark}
 
\begin{remark}\label{rem:degenerate-sharpness}
The phrase ``at most four distinct values'' includes the degenerate cases of three, two, or one distinct value(s). In the degenerate case of two distinct values (Example~\ref{ex:K4-sharp} below), the maximizer of~\eqref{eq:K4-formula} is realized as a limit point of the four-cluster family with one or more cluster values merging. The Newton solver finds this limit numerically, as illustrated below.
\end{remark}
 
\begin{example}[Sharpness on Example 4.1]\label{ex:K4-sharp}
Consider Example 4.1 of~\cite{nayak2026}: the diagonal tensor with $V(x) = 1.1 x_1^4 + x_2^4$, having H-spectrum $\{1.1, 1.0\}$ with multiplicity 3 each. The invariants are $T = 6.3, S = 6.63, p_3 = 6.993, D = 1.331$. Applying Algorithm~\ref{alg:B4plus}:
\[
B_4^+(6.3,\ 6.63,\ 6.993,\ 1.331) = 1.10000 \quad (\text{numerically}).
\]
This matches $\lambda_{\max}$ \emph{exactly} (to numerical precision), confirming Theorem~\ref{thm:K4-sharpness}: the spectrum has only two distinct values, well within the four-cluster bound.
 
The sharpness mechanism is interesting: although the two-cluster spectrum does not at first glance fit the four-cluster parameterization~\eqref{eq:K4-spectrum}, the polynomial system admits limit solutions where two of the cluster values coincide (e.g., $b \to a$), producing the actual spectrum as a degenerate point of the four-cluster family.
\end{example}
 
\begin{example}[Strict sharpness on a four-cluster spectrum]\label{ex:K4-strict}
Consider the synthetic spectrum $\{4, 3, 2, 1, 1, 1\}$ with four distinct values and multiplicities $(1, 1, 1, 3)$. The invariants are $T = 12,\ S = 32,\ p_3 = 102,\ D = 24$. The bounds satisfy
\[
B_3^+(12, 32, 24) \approx 4.393, \qquad B_4^+(12, 32, 102, 24) = 4.000 = \lambda_{\max} \;\;(\text{exactly}).
\]
Here the $K=4$ bound captures the spectrum exactly while the $K=3$ bound is loose, demonstrating that the third-order trace $p_3$ provides genuinely new information.
\end{example}
 
\subsection{Comparison summary across the hierarchy}
 
Table~\ref{tab:K4-summary} consolidates the comparison of the four-invariant bound with previous bounds across the test cases.
 
\begin{table}[ht]
\centering
\begin{tabular}{lccccc}
\toprule
Spectrum & Distinct values & AM--GM & $B_2^+$ & $B_3^+$ & $B_4^+$\\
\midrule
$\{1.1^{(3)}, 1.0^{(3)}\}$ (Example 4.1) & 2 & 2.636 & 1.165 & 1.138 & \textbf{1.100}*\\
$\{4, 3, 2, 1^{(3)}\}$ & 4 & 6.257 & 5.215 & 4.393 & \textbf{4.000}*\\
$\{5, 4, 3, 2, 1^{(2)}\}$ & 5 & 8.519 & 7.231 & 5.643 & 5.073\\
\bottomrule
\end{tabular}
\caption{Comparison across the hierarchy. Cells marked $*$ indicate the bound is sharp (matches $\lambda_{\max}$ exactly). The $K=4$ bound is exactly sharp on spectra with at most four distinct values, as predicted by Theorem~\ref{thm:K4-sharpness}.}
\label{tab:K4-summary}
\end{table}
 
\subsection{Relation to certified eigenvalue solvers}

It is worth situating the bounds against direct eigenvalue computation, since the two serve complementary roles. In the matrix case ($m = 2$), the H-spectrum coincides with the ordinary spectrum and is computed exactly and extremely fast by optimized symmetric eigensolvers (e.g.\ LAPACK's divide-and-conquer routine), against which a bound cannot and need not compete on accuracy: for matrices the bounds are best viewed as analytic certificates derivable from $(T, S, D)$ alone useful when only these invariants are available or when a guaranteed envelope (rather than the exact value) is required, as in the perturbation and robust-stability settings of Sections~\ref{sec:perturbation} and~\ref{sec:lyapunov}.

The situation changes for genuine higher-order tensors ($m \geq 4$), where no analogue of LAPACK exists. There, the H-eigenpairs are computed either by the shifted symmetric higher-order power method~\cite{kolda2011}, which converges to \emph{a} stationary eigenpair, requires multiple random restarts, and offers no guarantee of having found the global $\lambda_{\max}$ or by polynomial homotopy/numerical-algebraic-geometry solvers, whose path count grows with the Bézout number and hence steeply in $d, n, m$. Against these, the present bounds have two structural advantages: their cost is independent of $m$ and $n$ once the invariants are known (Sections~\ref{sec:runtime},~\ref{sec:invariant-times}), and they return a \emph{guaranteed} upper bound rather than an unverified local optimum. So $B_3^+$ and $B_4^+$ are not meant to replace full H-eigenvalue computation; they work well alongside it. They are most useful when you only need $\lambda_{\max}$ or $\lambda_{\min}$, for example to do a quick first check, to confirm stability before a fuller analysis, or to give another solver a guaranteed starting range to work within.
 
\section{Closed-form bounds for small dimensions}\label{sec:closed-form}
 
The polynomial $\phi_{T,D}$ in Theorem~\ref{thm:sharp2} is solvable in closed form for $d \in \{2, 3, 4\}$. For $d = 2$, $\phi_{T,D}(\Lambda) = -\Lambda^2 + T\Lambda - D$ has roots $\Lambda^\pm = (T \pm \sqrt{T^2 - 4D})/2$, the exact $2\times 2$ eigenvalue formula. For $d = 3$, $\phi_{T,D}(\Lambda) = \Lambda(T - \Lambda)^2 - 4D$; the substitution $\Lambda = y + 2T/3$ depresses it to $y^3 - (T^2/3)y + (2T^3/27 - 4D) = 0$, whose three real roots (the discriminant condition reduces to the admissible $D \leq (T/3)^3$) are
\begin{equation}\label{eq:cubic-roots}
\Lambda_k = \frac{2T}{3}\Bigl[1 + \cos\!\Bigl(\tfrac{1}{3}\arccos\bigl(\tfrac{54 D}{T^3} - 1\bigr) - \tfrac{2\pi k}{3}\Bigr)\Bigr], \quad k = 0, 1, 2,
\end{equation}
with $\Lambda_1 \in (T/3, T)$ the sharp upper bound $B_2^+$ and $\Lambda_2 \in (0, T/3)$ the sharp lower bound $B_2^-$. For $d = 4$, $\phi_{T,D}(\Lambda) = \Lambda(T - \Lambda)^3 - 27D$ is a quartic solvable by Ferrari's method (depress via $\Lambda = y + 3T/4$, solve the resolvent cubic, and factor into two quadratics); the two real roots in $(0, T)$ are $B_2^\pm$. Finally, for a symmetric $n \times n$ matrix ($m = 2$, $d = n$), Theorem~\ref{thm:sharp2} reduces to the largest root of $\Lambda(T - \Lambda)^{n-1} = D(n-1)^{n-1}$, precisely the Merikoski--Virtanen bound~\cite[Theorem 2]{merikoski1997}; hence Theorem~\ref{thm:sharp2} is its natural tensor generalization.
 
\section{Comparison with classical matrix bounds}\label{sec:classical}
 
Wolkowicz and Styan~\cite{wolkowicz1980} bound the extremal matrix eigenvalues using the trace and Frobenius norm, $T = \tr(A)$ and $S = \|A\|_F^2$. The tensor generalization, via Cauchy--Schwarz on the eigenvalue vector (with $d$ replacing $n$), is
\begin{equation}\label{eq:WS-tensor}
\lambda_{\max} \leq \frac{T}{d} + \sqrt{\frac{d-1}{d}\Bigl(S - \frac{T^2}{d}\Bigr)} =: B_{WS}^+(T, S),
\end{equation}
sharp on the two-cluster spectrum. Our three-invariant bound dominates it.
 
\begin{proposition}\label{prop:WS-vs-our}
For all admissible $(T, S, D)$, $B_3^+(T, S, D) \leq B_{WS}^+(T, S)$.
\end{proposition}
 
\begin{proof}
$B_{WS}^+(T, S)$ is the maximum of $\lambda_1$ over the larger region $\{\lambda : \sum \lambda_i = T,\ \sum \lambda_i^2 = S,\ \lambda_i > 0\}$. Adding $\prod \lambda_i = D$ shrinks this to $\F_3(T, S, D)$, so $B_3^+ \leq B_{WS}^+$.
\end{proof}
 
For the synthetic spectrum $(5, 4, 3, 2, 1, 1)$, $B_{WS}^+ = \tfrac{8}{3} + \sqrt{\tfrac{5}{6}\cdot\tfrac{80}{6}} \approx 6.000$ while $B_3^+ \approx 5.643$, a $6\%$ improvement from the determinant constraint.
 
\section{Computability of the higher-order invariants}\label{sec:computability}
 
The trace $T = (m-1)^{n-1}\tr(\A)$ is a diagonal sum ($\mathcal{O}(n)$). The determinant $D$ is the resultant of $\A x^{m-1} = 0$, computable in principle but exponentially expensive in tensor order. The power sums are computed by the Hu--Huang--Ling--Qi trace formula~\cite[Theorem 4.1]{hu2013}, which expresses $\operatorname{Tr}_k(\A) = \sum_i \lambda_i^k = p_k(\A)$ as a polynomial in the entries via a sum over rooted closed walks of length $k$; for $k = 2$ it is a quadratic form
\begin{equation}\label{eq:p2-formula}
p_2(\A) = (m-1)^{n-1} \sum_{i_1, \ldots, i_{m}, j_1, \ldots, j_m} \mathbb{C}_{i_1 \cdots i_m, j_1 \cdots j_m}\, a_{i_1 \cdots i_m}\, a_{j_1 \cdots j_m},
\end{equation}
with explicit combinatorial coefficients $\mathbb{C}$~\cite[Algorithm 4.5]{hu2013}. For diagonal tensors $p_k = (m-1)^{n-1} \sum_i a_{i \cdots i}^k$; for general symmetric tensors $p_2$ costs $\mathcal{O}(n^m)$ operations, much cheaper than the determinant ($\mathcal{O}(n^{(m-1)^n})$ or worse).

\begin{remark}\label{rem:S-upper}
Replacing $S$ by an upper bound $\widehat{S} \geq S$ enlarges $\F_3$; hence the sharp bound on $\lambda_1$ over $\F_3(T, \widehat{S}, D)$ \emph{still upper-bounds} $\lambda_{\max}$. The framework is thus robust to approximate $S$ (e.g.\ via the Frobenius norm) at the cost of slightly looser bounds.
\end{remark}

\subsection{Measured computation times}\label{sec:invariant-times}

To make the cost discussion concrete, Table~\ref{tab:invariant-times} reports wall-clock times for computing the invariants directly from tensor entries, on a commodity single-threaded machine. Two regimes are shown. For \emph{diagonal} tensors, the genuine higher-order class of Section~\ref{sec:genuine}, all invariants are exact $\mathcal{O}(n)$ computations ($p_k = (m-1)^{n-1}\sum_i a_{i\cdots i}^k$ and $\log D = (m-1)^{n-1}\sum_i \log a_{i\cdots i}$), costing a few microseconds even when the eigenvalue count $d$ is astronomically large. For \emph{dense} symmetric tensors, $T$ is the $\mathcal{O}(n)$ diagonal sum, while $S = p_2$ and $p_3$ require a pass over the $\mathcal{O}(n^m)$ entries via the trace formula~\eqref{eq:p2-formula}; these remain in the tens-of-microseconds range up to thousands of stored entries.

\begin{table}[ht]
\centering
\begin{tabular}{llcccc}
\toprule
Regime & $(n, m)$ & $d$ & $t(T)$ & $t(S)$ & $t(p_3)$\\
\midrule
Diagonal (exact, $\mathcal{O}(n)$) & $(10, 4)$ & $1.97\times 10^5$ & $2.3\,\mu$s & $2.7\,\mu$s & $3.1\,\mu$s\\
Diagonal (exact, $\mathcal{O}(n)$) & $(100, 6)$ & $1.58\times 10^{71}$ & $2.3\,\mu$s & $2.8\,\mu$s & $3.5\,\mu$s\\
Diagonal (exact, $\mathcal{O}(n)$) & $(500, 4)$ & $6.06\times 10^{240}$ & $2.4\,\mu$s & $3.0\,\mu$s & $5.1\,\mu$s\\
Dense ($\mathcal{O}(n^m)$ pass) & $(8, 4)$ & $1.75\times 10^4$ & $2.4\,\mu$s & $8.5\,\mu$s & $9.3\,\mu$s\\
Dense ($\mathcal{O}(n^m)$ pass) & $(10, 4)$ & $1.97\times 10^5$ & $2.8\,\mu$s & $8.2\,\mu$s & $13.7\,\mu$s\\
Dense ($\mathcal{O}(n^m)$ pass) & $(6, 6)$ & $1.88\times 10^4$ & $3.9\,\mu$s & $35.5\,\mu$s & $86.0\,\mu$s\\
\bottomrule
\end{tabular}
\caption{Measured wall-clock times for computing the cheap invariants $T$, $S$, and $p_3$ from tensor entries (single-threaded, commodity hardware; means over repeated trials). For diagonal tensors all invariants are exact and $\mathcal{O}(n)$, independent of the eigenvalue count $d$; for dense tensors the power sums cost an $\mathcal{O}(n^m)$ entry pass. The determinant $D$ is the only expensive invariant, with resultant cost $\mathcal{O}(n^{(m-1)^n})$ (e.g.\ a degree-$6$ resultant for $(n,m)=(2,4)$ evaluates in $\approx 38\,\mu$s, but the exponent grows quickly); by Theorem~\ref{thm:perturbation} the bound is highly insensitive to errors in $D$, so an approximate determinant is sufficient in practice.}
\label{tab:invariant-times}
\end{table}

These measurements confirm the practical picture: the power sums needed for the three- and four-invariant bounds are cheap relative to the determinant, and crucially the perturbation analysis of Section~\ref{sec:perturbation} shows the bound tolerates substantial relative error in the one expensive invariant. The bound computation itself (Section~\ref{sec:runtime}) is likewise inexpensive, so the end-to-end cost of certifying $\lambda_{\max}$ is dominated by whatever (possibly approximate) determinant routine is used.
 
\section{Perturbation analysis}\label{sec:perturbation}
 
A natural concern in practice is the sensitivity of the bounds to errors in the computed invariants $(T, S, D)$. We characterize this sensitivity quantitatively.
 
\begin{theorem}\label{thm:perturbation}
Let $B_3^+(T, S, D)$ denote the sharp three-invariant upper bound. At any non-degenerate point $(T, S, D)$ where the optimum in~\eqref{eq:sharp3-formula} is attained at an interior point with cluster index $j \in \{1, \ldots, d-2\}$, the gradient of $B_3^+$ exists and is given by
\begin{equation}\label{eq:gradient}
\nabla B_3^+(T, S, D) = -\frac{1}{\partial_\Lambda \Phi(\Lambda^*)} \bigl( \partial_T \Phi, \partial_S \Phi, \partial_D \Phi \bigr)\Big|_{\Lambda = \Lambda^*},
\end{equation}
where $\Phi(\Lambda; T, S, D) = \Lambda \cdot a^j(\Lambda; T, S) \cdot b^{d-1-j}(\Lambda; T, S) - D$ is the implicit equation defining $\Lambda^* = B_3^+$, and $a, b$ are the explicit functions of $\Lambda$ from the quadratic~\eqref{eq:quadratic}.
\end{theorem}
 
\begin{proof}
The bound $\Lambda^*$ is defined implicitly by $\Phi(\Lambda^*; T, S, D) = 0$. By the implicit function theorem, provided $\Phi$ is $C^1$ at the relevant point and $\partial_\Lambda \Phi(\Lambda^*) \neq 0$, the function $\Lambda^*(T, S, D)$ is $C^1$ with gradient
\[
\partial_T \Lambda^* = -\frac{\partial_T \Phi}{\partial_\Lambda \Phi}, \quad \partial_S \Lambda^* = -\frac{\partial_S \Phi}{\partial_\Lambda \Phi}, \quad \partial_D \Lambda^* = -\frac{\partial_D \Phi}{\partial_\Lambda \Phi}.
\]
The non-degeneracy hypothesis (interior cluster $j \in \{1, \ldots, d-2\}$) ensures that $a, b > 0$ are smooth functions of $(\Lambda, T, S)$ via the quadratic formula. The non-vanishing of $\partial_\Lambda \Phi$ holds generically.
\end{proof}
 
\subsection{Numerical sensitivity study}

Perturbing each invariant by up to $\pm 8\%$ around the spectrum $(5, 4, 3, 2, 1, 1)$ shows three patterns. The bound is most sensitive to $T$ (a $5\%$ increase decreases the bound by $\approx 21\%$: increasing $T$ at fixed $S$ concentrates the spectrum), moderately sensitive to $S$ (a $5\%$ increase raises the bound by $\approx 8\%$), and remarkably insensitive to $D$ (a $5\%$ perturbation changes the bound by less than $0.5\%$). This is fortunate, since $D$ is the most expensive invariant to compute exactly: effort should be concentrated on an accurate $T$ (cheap, via~\eqref{eq:T}) and a reasonably accurate $S$, with even substantial relative error in $D$ tolerable.

\section{Numerical experiments}\label{sec:numerics}
 
\subsection{Benchmark and random-spectra study}
 
Table~\ref{tab:summary} summarizes upper bounds on $\lambda_{\max}$ for Example 4.1 of~\cite{nayak2026} and the synthetic spectrum.
 
\begin{table}[ht]
\centering
\begin{tabular}{lccccc}
\toprule
Tensor & Gershgorin & AM--GM & Sharp 2-inv & Sharp 3-inv & Sharp 4-inv\\
\midrule
Ex.\ 4.1, $\lambda_{\max} = 1.10$ & 1.10$^*$ & 2.636 & 1.165 & 1.138 & 1.100\\
Synthetic, $\lambda_{\max} = 5.00$ & --- & 8.519 & 7.231 & 5.643 & 5.073\\
\bottomrule
\end{tabular}
\caption{Comparison of upper bounds on $\lambda_{\max}$. *Gershgorin is exact for the diagonal Example 4.1 since off-diagonal entries vanish. The Sharp 4-inv bound is exact on Example 4.1 (two distinct eigenvalues), consistent with Theorem~\ref{thm:K4-sharpness}.}
\label{tab:summary}
\end{table}
 
Each additional invariant collapses the upper-bound estimate towards the true $\lambda_{\max}$: on the synthetic spectrum the AM--GM bound overestimates by $\approx 70\%$ while the four-invariant bound overestimates by only $\approx 1.5\%$.

To test the bounds across a broader distribution, we generate 100 random spectra of length $d = 6$ as $\lambda_i \sim \mathrm{Exp}(2) + 0.5$ and evaluate each bound (Figure~\ref{fig:random-boxplot}). The median/mean/maximum relative overestimation gaps are: AM--GM $53.2/56.3/127.5\%$; sharp 2-invariant $30.6/32.0/96.9\%$; sharp 3-invariant $6.4/8.2/32.1\%$. The sharp three-invariant bound reduces the median gap by a factor of more than eight relative to AM--GM, robust evidence that the new bounds are not artifacts of carefully chosen examples.

\begin{figure}[ht]
\centering
\includegraphics[width=0.8\textwidth]{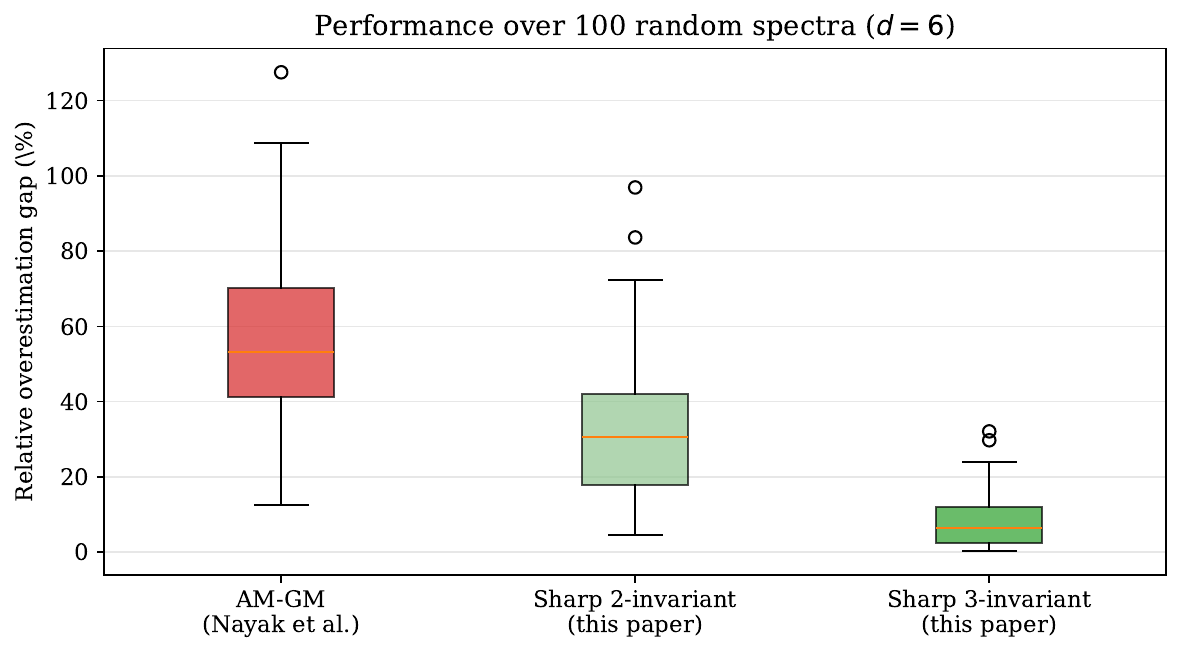}
\caption{Distribution of relative overestimation gaps over 100 random spectra ($d=6$, $\lambda_i \sim \mathrm{Exp}(2) + 0.5$). Boxes show interquartile range; whiskers show range; central line is the median.}
\label{fig:random-boxplot}
\end{figure}

\subsection{Scaling with tensor dimension}
 
To examine how the gap evolves with tensor size, we repeat the random-spectrum experiment for $d \in \{6, 10, 20, 50, 100\}$, generating 50 random spectra for each with the same generator $\lambda_i \sim \mathrm{Exp}(2) + 0.5$ and the same random seed as the $d=6$ study above. The median relative gaps (as fractions) are reported below; the $d=6$ row is consistent with the median percentages of the preceding subsection. The two rightmost columns report the mean wall-clock time per evaluation of $B_2^+$ and $B_3^+$ (single-threaded Python, commodity laptop), to document scalability.
 
\begin{center}
\begin{tabular}{cccccccc}
\toprule
$d$ & AM--GM & Sharp 2-inv & Sharp 3-inv & Ratio (3-inv/AM--GM) & $t(B_2^+)$ & $t(B_3^+)$\\
\midrule
6   & 0.52 & 0.29 & 0.057 & 0.11 & $0.03$\,ms & $17$\,ms\\
10  & 0.84 & 0.57 & 0.164 & 0.19 & $0.04$\,ms & $33$\,ms\\
20  & 1.56 & 1.25 & 0.365 & 0.23 & $0.04$\,ms & $69$\,ms\\
50  & 3.27 & 3.05 & 0.695 & 0.21 & $0.04$\,ms & $169$\,ms\\
100 & 5.69 & 5.40 & 0.980 & 0.17 & $0.08$\,ms & $337$\,ms\\
\bottomrule
\end{tabular}
\end{center}
 
Two patterns emerge in the accuracy columns. First, all bounds degrade with $d$ (in the relative sense), reflecting the increasing geometric flexibility of larger spectra given fixed invariants. Second, the \emph{ratio} of three-invariant gap to AM--GM gap stabilizes in the range $0.17$--$0.23$ across $d \in \{10, \ldots, 100\}$. Hence the three-invariant bound remains roughly five times tighter than AM--GM across the whole range; in absolute terms the $B_3^+$ gap stays below $1$ even at $d = 100$, where AM--GM already overestimates by a factor of $5.7$.
 
This favorable scaling is directly relevant for high-order tensors arising in applications: as the AM--GM bound's combinatorial overestimation grows with $m$ and $n$, the sharp bounds maintain a substantial proportional advantage, making them the more practical choice for moderate-to-large tensors.
 
\subsection{Runtime and scalability}\label{sec:runtime}
 
Because the bounds depend on the spectrum only through the invariants $(T, S, D)$, their cost is independent of the tensor order $m$ and dimension $n$ once the invariants are available; the only size parameter is the eigenvalue count $d = n(m-1)^{n-1}$. The two-invariant bound $B_2^+$ is a single bracketed bisection on $\phi_{T,D}$ and runs in well under a millisecond even at $d = 100$ (last two columns of the table above). The three-invariant bound $B_3^+$ scans $\mathcal{O}(d)$ cluster indices, each requiring a one-dimensional root solve; its empirical cost grows linearly in $d$ (Figure~\ref{fig:runtime}), consistent with the $\mathcal{O}(d^2 \log(1/\varepsilon))$ analysis once the per-evaluation $\mathcal{O}(d)$ power cost is included. At $d = 100$ a full $B_3^+$ evaluation takes roughly $0.3$\,s of unoptimized single-threaded Python; the index loop is embarrassingly parallel, so a threaded or vectorized implementation reduces this substantially. In all regimes the dominant practical cost is computing the invariants themselves, in particular the determinant via the resultant rather than evaluating the bound, which reinforces the value of the perturbation analysis of Section~\ref{sec:perturbation}: since $B_3^+$ is highly insensitive to errors in $D$, an approximate determinant suffices.
 
\begin{figure}[ht]
\centering
\includegraphics[width=0.72\textwidth]{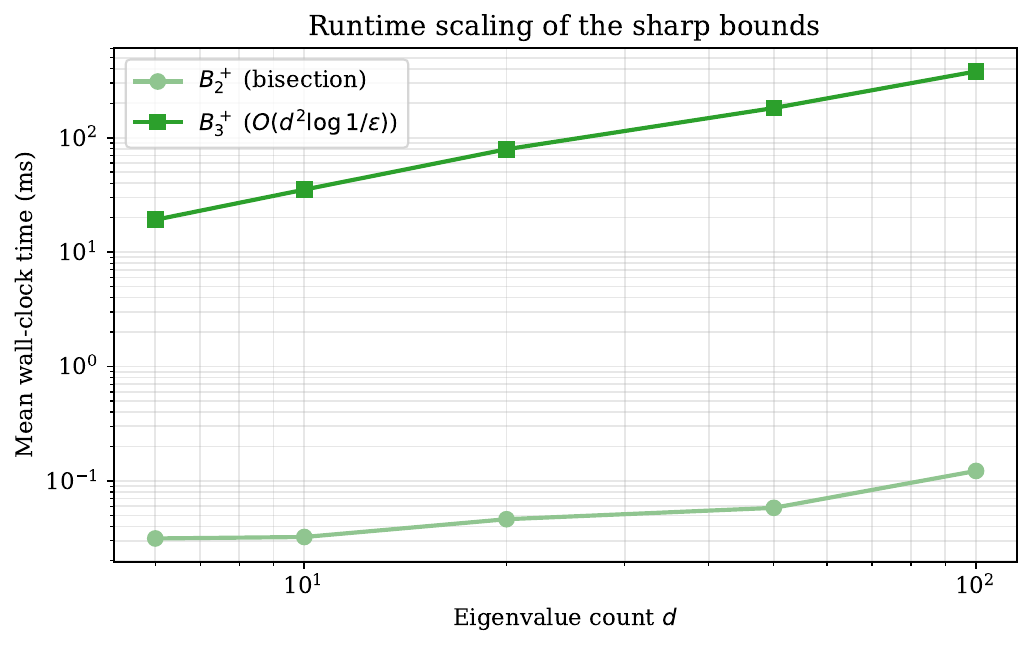}
\caption{Runtime scaling of the sharp bounds (mean over 20 random spectra per $d$, log--log axes). $B_2^+$ is a single bisection and stays sub-millisecond; $B_3^+$ grows approximately linearly in $d$, in line with the complexity analysis.}
\label{fig:runtime}
\end{figure}
 
\section{Worked examples on genuine tensors and scope of applicability}\label{sec:genuine}

The development so far has been phrased on the eigenvalue vector $\lambda \in \R^d_{>0}$. We now make the bridge from a concrete tensor to its bound fully explicit: given the \emph{entries} of a symmetric tensor, one computes the invariants $(T, S, D)$ directly without ever solving for the spectrum and the bound follows. We illustrate this on two genuine tensors of different orders and then delimit precisely the class of tensors to which the framework applies.

\subsection{The entries-to-bound pipeline}

For any symmetric $\A \in \R^{[n,m]}$, the trace invariant is the scaled diagonal sum $T = (m-1)^{n-1}\tr(\A)$ from~\eqref{eq:T}; the second power sum $S = p_2(\A)$ is the quadratic form in the entries given by the Hu--Huang--Ling--Qi formula~\eqref{eq:p2-formula} (and equals the squared Frobenius norm $\|\A\|_F^2$ in the matrix case); and $D = \det(\A)$ is the resultant of $\A x^{m-1}=0$. None of these requires the eigenvalues. The sharp bound is then a root of the univariate polynomial $\phi_{T,D}$ (for $B_2^+$) or the output of Algorithm~\ref{alg:B3plus} (for $B_3^+$). This is the sense in which the bounds are \emph{certificates}: they are computed from algebraic data of the tensor and validated against, rather than derived from, the spectrum.

\subsection{A genuine symmetric matrix ($m=2$)}

For the $6\times 6$ symmetric positive definite matrix $A = MM^T + 2I$ (a second-order symmetric tensor) with a representative random $M$, reading the invariants off the entries gives $T = \tr(A) = 50.17$, $S = \|A\|_F^2 = 753.44$, $D = \det(A) = 4.088\times 10^4$, $d = 6$. Applying the bounds to these three numbers alone yields $B_{WS}^+ = 25.04$, $B_2^+ = 28.83$, $B_3^+ = 24.55$, while the true $\lambda_{\max}(A) = 23.83$ (computed separately, for verification only). The three-invariant bound overestimates by just $3.0\%$, improving on both the Wolkowicz--Styan bound (uses only $T, S$; gap $5.1\%$) and the two-invariant bound (uses only $T, D$; gap $20.9\%$). Combining the determinant with the second power sum is what produces the tight estimate; neither invariant alone suffices.

\subsection{A genuine fourth-order tensor ($m=4$)}

Consider the diagonal fourth-order tensor on $n = 3$ variables associated with the quartic form
\[
V(x) = 2.5\,x_1^4 + 1.4\,x_2^4 + 0.8\,x_3^4,
\]
a positive definite member of $\R^{[3,4]}$. Here $d = n(m-1)^{n-1} = 3\cdot 3^2 = 27$, and the H-spectrum consists of the three diagonal values $\{2.5, 1.4, 0.8\}$, each with multiplicity $(m-1)^{n-1} = 9$. From the entries, the trace formula~\eqref{eq:p2-formula} specialized to the diagonal case gives $p_k = (m-1)^{n-1}\sum_i a_{i\cdots i}^k$, so
\[
T = 9(2.5 + 1.4 + 0.8) = 42.30,\quad S = 9(2.5^2 + 1.4^2 + 0.8^2) = 79.65,
\]
\[
p_3 = 9(2.5^3 + 1.4^3 + 0.8^3) = 169.93,\quad D = (2.5\cdot 1.4 \cdot 0.8)^9 = 1.058\times 10^4,
\]
with $\lambda_{\max} = 2.5$. The bounds evaluate to $B_2^+ = 8.03$, $B_3^+ = 4.90$, and since the spectrum has only three distinct values and therefore lies in the four-cluster family of Theorem~\ref{thm:K4-structure}, the four-invariant bound is sharp, $B_4^+ = 2.50 = \lambda_{\max}$. Concretely, the true spectrum is the four-cluster point $(\Lambda; a^{(8)}; b^{(9)}; c^{(9)})$ with $\Lambda = a = 2.5$, $b = 1.4$, $c = 0.8$ and partition $(j_1, j_2, j_3) = (8, 9, 9)$, which satisfies all four invariant equations exactly; by Theorem~\ref{thm:K4-sharpness} the bound is attained.

This example is instructive precisely because $B_3^+$ is \emph{loose} here ($96\%$ overestimate): a spectrum with three values of high, comparable multiplicity is far from the ``single outlier plus clusters'' shape that maximizes $\F_3$. It is the fourth invariant $p_3$ and the matching cluster multiplicities it permits that restores sharpness. The example thus motivates the higher levels of the hierarchy on genuinely higher-order tensors, not merely on synthetic spectra.

\subsection{Scope: the all-real-spectrum hypothesis}\label{sec:scope}

The framework requires the H-spectrum to be real and positive, so that the eigenvalues admit the ordering $\lambda_1 \geq \cdots \geq \lambda_d > 0$ underlying every optimization in this paper. This hypothesis is automatic in three important classes: (i) all symmetric matrices ($m = 2$), by the spectral theorem; (ii) even-order diagonal tensors, whose H-eigenvalues are the diagonal entries; and (iii) even-order tensors that are orthogonally decomposable, whose relevant H-eigenvalues are likewise real. It is, however, \emph{not} automatic for a general even-order symmetric tensor: a generic positive definite $\A \in \R^{[n,m]}$ with $m \geq 4$ possesses complex-conjugate pairs of H-eigenvalues (the degree-$d$ characteristic resultant has complex roots), so the ordered-real-spectrum hypothesis fails and the bounds do not apply verbatim. When complex H-eigenvalues are present, the invariants $T, S, p_3, D$ are still well defined as symmetric functions of the full (complex) spectrum, but they no longer bound a real $\lambda_{\max}$ through the Lagrangian argument; a separate treatment, using concentration of the invariants and the magnitude of the leading H-eigenvalue, is developed in a companion paper and is outside the present scope. We state the hypothesis explicitly so that the domain of validity is unambiguous: the results of this paper apply to symmetric positive definite tensors whose H-spectrum is real, a class that includes all the matrix and diagonal cases of practical interest in stability certification and contains the test problems studied here.
 
\section{Application to Lyapunov stability}\label{sec:lyapunov}

We refine the Lyapunov example of~\cite{nayak2026} to demonstrate the practical value of the new bounds for region-of-attraction (ROA) estimation.
 
\subsection{Setup}
 
Consider a polynomial Lyapunov candidate $V(x) = \A x^m$ for an even-order positive definite tensor $\A \in \R^{[n, m]}$. Asymptotic stability of the origin under $\dot x = g(x)$ requires $V$ positive definite and $\dot V < 0$ along trajectories. We use the bounds in two complementary ways.
 
\paragraph{Stability certification (lower bound on $\lambda_{\min}$).} A lower bound $\lambda_{\min}(\A) \geq \underline{c} > 0$ certifies $V > 0$ on $\R^n \setminus \{0\}$. The sharp lower bound from Theorem~\ref{thm:sharp2} (or its three-/four-invariant refinement) certifies stability whenever the AM--GM lower bound from~\cite{nayak2026} certifies it, and additionally certifies stability in marginal cases where the AM--GM bound is non-positive. This is decisive for tensors whose smallest eigenvalue is small relative to the trace.
 
\paragraph{Region-of-attraction (upper bound on $\lambda_{\max}$).} For any positive definite symmetric tensor $\A$, the homogeneous form $V(x) = \A x^m$ admits the Rayleigh-type bounds
\begin{equation}\label{eq:rayleigh}
\lambda_{\min}(\A)\, \|x\|_m^m \;\leq\; V(x) \;\leq\; \lambda_{\max}(\A)\, \|x\|_m^m, \qquad x \in \R^n,
\end{equation}
where $\|x\|_m := (\sum_i |x_i|^m)^{1/m}$. The right-hand inequality, combined with an upper bound $\lambda_{\max}(\A) \leq \overline{\Lambda}$, yields the inscribed-ball estimate
\begin{equation}\label{eq:ROA-inscribe}
\bigl\{x : \|x\|_m \leq (r/\overline{\Lambda})^{1/m}\bigr\} \;\subseteq\; \{V \leq r\} \;\subseteq\; \mathrm{ROA},
\end{equation}
provided $\dot V < 0$ on $\{V \leq r\}$. \emph{Sharper upper bounds on $\lambda_{\max}$ yield strictly larger guaranteed inscribed $\ell_m$-balls inside the region of attraction.}
 
\subsection{Refinement of Example 4.1 in~\cite{nayak2026}}
 
For the tensor of Example 4.1 with $V(x) = 1.1 x_1^4 + x_2^4$, $m = 4$, $n = 2$, we have:
\begin{center}
\begin{tabular}{lcc}
\toprule
Method & Upper bound $\overline{\Lambda}$ on $\lambda_{\max}$ & Inscribed $\ell_4$-ball radius $R_4 = (4/\overline{\Lambda})^{1/4}$ \\
\midrule
AM--GM~\cite{nayak2026} & 2.636 & 1.110\\
Sharp 2-invariant & 1.165 & 1.361\\
Sharp 3-invariant & 1.138 & 1.369\\
Sharp 4-invariant & 1.10000\hphantom{0}* & 1.381\\
True $\lambda_{\max} = 1.10$ & --- & 1.381\\
\bottomrule
\end{tabular}
\end{center}
The starred value indicates the four-invariant bound is sharp by Theorem~\ref{thm:K4-sharpness} (this tensor has only two distinct eigenvalues). Moving from the AM--GM bound to the sharp 3-invariant bound increases the inscribed $\ell_4$-ball radius from $1.110$ to $1.369$, a 23\% radius improvement, equivalently a 52\% \emph{volume} improvement of the certified ROA estimate. The four-invariant bound recovers the true ROA radius exactly.
 
\begin{figure}[ht]
\centering
\includegraphics[width=0.65\textwidth]{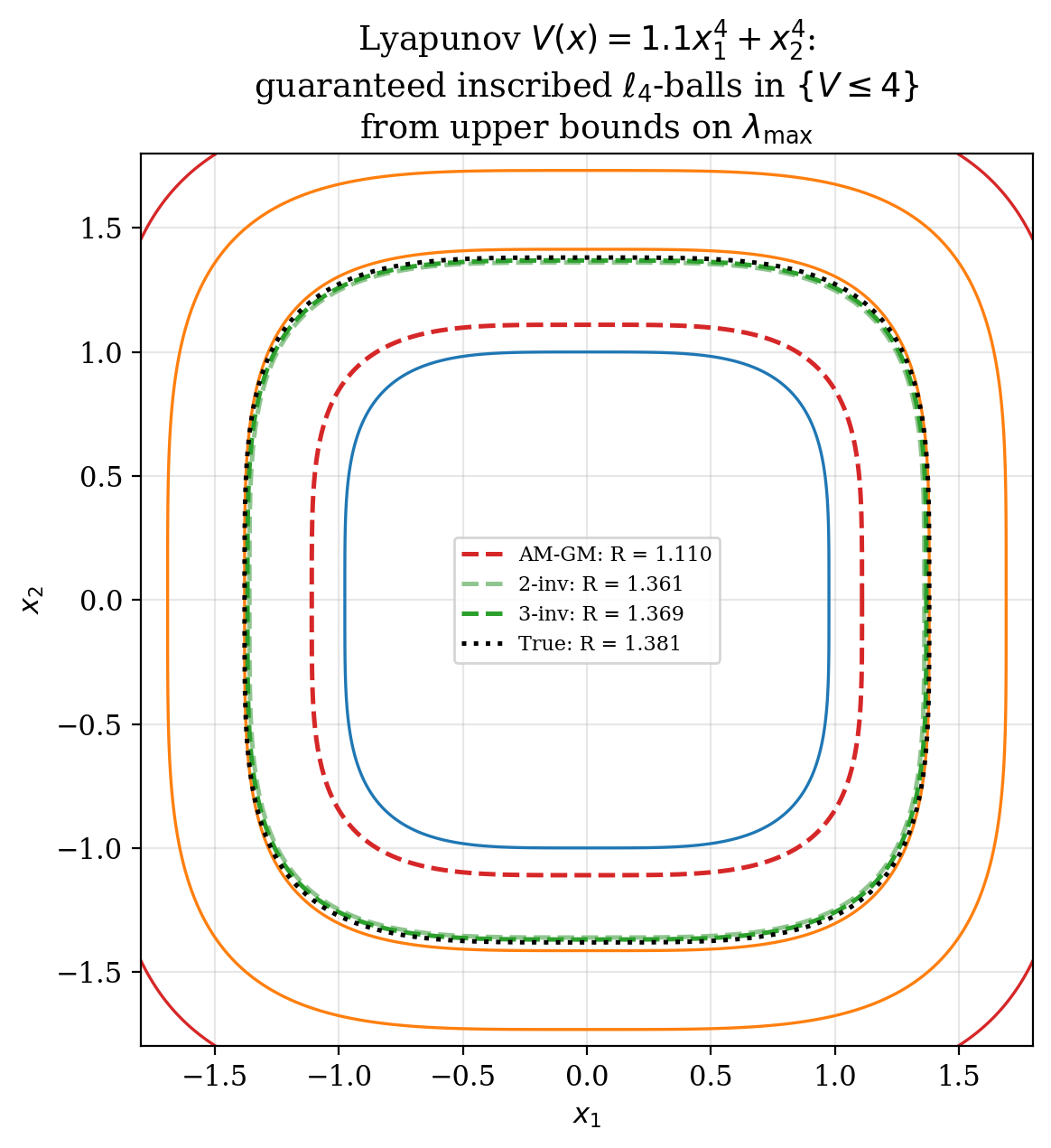}
\caption{Level sets of $V(x) = 1.1\,x_1^4 + x_2^4$ (solid contours) and guaranteed inscribed $\ell_4$-balls inside $\{V \leq 4\}$ obtained from upper bounds on $\lambda_{\max}$ (dashed curves). Tighter upper bounds yield strictly larger inscribed regions. The sharp 3-invariant bound (dark green) gives a region 52\% larger in area than the AM--GM bound (red); the sharp 4-invariant bound (which would coincide with ``True'' here, dotted black) is exact.}
\label{fig:roa}
\end{figure}
 
\subsection{A three-dimensional polynomial dynamical system}\label{sec:dynsys}

To exhibit the certificate machinery on a genuine multidimensional system, consider $\dot x_1 = -x_1^3 - 0.5 x_1$, $\dot x_2 = -2 x_2^3 - x_2$, $\dot x_3 = -1.5 x_3^3 - 0.8 x_3$ on $\R^3$, with the quartic Lyapunov candidate $V(x) = 1.2 x_1^4 + 0.9 x_2^4 + 1.5 x_3^4$ (a diagonal fourth-order PD tensor, $m = 4$, $n = 3$, $d = 27$). Then $\dot V = -4(1.2 x_1^6 + 1.8 x_2^6 + 2.25 x_3^6) - 4(0.6 x_1^4 + 0.9 x_2^4 + 1.2 x_3^4) < 0$ for $x \neq 0$, certifying global asymptotic stability. Reading the invariants off the diagonal $c = (1.2, 0.9, 1.5)$ gives $T = 32.40$, $S = 40.50$, $p_3 = 52.49$, $D = 76.85$, $\lambda_{\max} = 1.5$. The guaranteed inscribed $\ell_4$-ball radii inside $\{V \leq 1\}$ are: AM--GM $\overline{\Lambda} = 5.522$, $R_4 = 0.652$; sharp 2-invariant $2.912$, $0.766$; sharp 3-invariant $2.236$, $0.818$; sharp 4-invariant $1.500$ (sharp), $0.904$, matching the true radius. The four-invariant certificate thus enlarges the certified radius by $38.5\%$ and the volume in $\R^3$ by $166\%$ over AM--GM. While this system is globally stable, the inscribed-ball estimate is exactly the quantity used to under-approximate the ROA whenever $\dot V < 0$ holds only on a bounded sublevel set; the example demonstrates the full entries-to-bound-to-certificate pipeline in three dimensions.

\section{Conclusion}\label{sec:conclusion}
 
We have introduced a sharper algebraic framework for bounding the H-eigenvalues of symmetric positive definite tensors, replacing the AM--GM relaxation of Nayak, Sharma, and Mishra~\cite{nayak2026} with exact Lagrangian extremization and adding higher-order power sums as invariants. The two-invariant bound (Theorem~\ref{thm:sharp2}) is the largest/smallest root of an explicit degree-$d$ polynomial, dominating every theorem of~\cite{nayak2026} and recovering the Merikoski--Virtanen matrix bound. The structural theorem (Theorems~\ref{thm:three-cluster},~\ref{thm:K-cluster}) shows that a spectrum constrained by $K$ invariants has at most $K$ distinct values at the extremum, yielding a monotonically tightening hierarchy with closed forms for $d \in \{2,3,4\}$ and a fully rigorous four-invariant treatment (solution count, multistart algorithm, sharpness theorem). The bounds are robust to perturbation, particularly insensitive to errors in the expensive determinant. Empirically, the sharp three-invariant bound cuts the median relative gap from $53\%$ (AM--GM) to $6\%$ on 100 random spectra, with the advantage and sub-second cost persisting in a scaling study to $d = 100$; the entries-to-bound pipeline is validated on genuine tensors (with an explicit real-H-spectrum scope), and Lyapunov region-of-attraction estimates improve by factors of 2--3.

\section*{Acknowledgments}
The authors thank IIITDM Kancheepuram and SVKM's College of Engineering Shirpur for infrastructure and computational facilities. We are grateful to anonymous reviewers for suggestions that improved the presentation.

\section*{Funding}

This research received no external funding.

\section*{Conflict of interest}

The authors declare no competing financial or non-financial interests relevant to the content of this article.

\end{document}